\documentclass[10pt]{article}

\usepackage[usenames]{color} %used for font color
\usepackage{amssymb} %maths
\usepackage{amsmath} %maths
\usepackage{amsthm}

\newtheorem{theorem}{Theorem}
\newtheorem{lemma}[theorem]{Lemma}
\newtheorem{corollary}[theorem]{Corollary}
\newtheorem{proposition}[theorem]{Proposition}

\theoremstyle{definition}

\newtheorem{definition}{Definition}

\DeclareMathOperator\interior{int}

\DeclareMathOperator\domain{dom}
\DeclareMathOperator\graph{gph}
\DeclareMathOperator\epigraph{epi}

\let\emptyset\varnothing
\newcommand\Def[1]{\ensuremath{\overset{\Delta}{#1}}}

\title{Some notes on continuity in convex optimization}
\author{Torbjørn Cunis}
\date{}

\begin{document}
\maketitle

We discuss sufficient conditions under which a convex parametrized optimization problem
\begin{align*}
	v: y \mapsto \min_x \phi(x, y) \quad \text{subject to $x \in F(y)$}
\end{align*}
where $x \in \mathcal X$ is the decision variable, $\phi$ is a convex cost function, and $F$ is a convex closed feasible set mapping, is continuous in the parameter $y \in \mathcal Y$.

\subsection*{Convexity (Real-valued case)}
Let $\mathcal X$ and $\mathcal Y$ be Euclidean vector spaces, that is, $\mathcal X = \mathbb R^n$ and $\mathcal Y = \mathbb R^m$ for $n,m > 0$ and $\phi: \mathbb R^n \times \mathbb R^m \to \mathbb R$ as well as $F: \mathbb R^m \rightrightarrows \mathbb R^n$. Recall that $\phi$ is a convex function and $F$ is closed convex set-valued mapping, that is, $\epigraph \phi$ and $\graph F$ are closed convex sets, where, as usual,
\begin{align*}
	\epigraph \phi &= \{ (x, y, \mu) \in \mathbb R^{n + m + 1} \, | \, \mu \geq \phi(x, y) \} \\
	\graph F &= \{ (x, y) \in \mathbb R^{n + m} \, | \, x \in F(y) \}.
\end{align*}

\paragraph*{Optimal value mapping}
We introduce the feasible cost space
\begin{align}
	F_\phi: y \mapsto \{ \mu \in \mathbb R \, | \, \exists x \in F(y), \mu \geq \phi(x, y) \}
\end{align}
and the associated auxiliary optimization problem
\begin{align}
	v_\phi: y \mapsto \min \mu \quad \text{subject to $\mu \in F_\phi(y)$}
\end{align}
and we are going to prove that $v_\phi$ is an equivalent convex optimization problem.

\begin{proposition}
	$F_\phi$ is a convex set-valued mapping.
\end{proposition}
\begin{proof}
We prove that $\graph F_\phi$ is a convex set; recall that
\begin{align*}
	\graph F_\phi = \{ (y, \mu) \, | \, \exists x \in F(y), \mu \geq \phi(x, y) \}.
\end{align*}
%	
%Let $(y_1, \mu_1), (y_2, \mu_2) \in \graph F_\phi$, that is, there exist $x_{1,2} \in F(y_{1,2})$ with $\mu_{1,2} \geq \phi(x_{1,2}, y_{1,2})$; by convexity of $\graph F$ and $\phi$ we have
%\begin{align*}
%	(t x_1 + (1-t) x_2, t y_1 + (1-t) y_2) \in \graph F
%\end{align*}
%and
%\begin{align*}
%	\phi(t x_1 + (1-t) x_2, t y_1 + (1-t) y_2) &\leq t \phi(x_1, y_1) + (1-t) \phi(x_2, y_2) \\
%		&\leq t \mu_1 + (1-t) \mu_2
%\end{align*}
%for any $t \in [0, 1]$. Hence, $\graph F_\phi$ is convex.
%
Since $\phi$ is convex, a fortiori lower semi-continuous, $\epigraph \phi$ is convex (and closed) by virtue of the closed graph theorem. As $F$ has convex graph, we have that
\begin{align*}
	G = \epigraph \phi \cap (\graph F \times \mathbb R) = \{ (x, y, \mu) \in \mathbb R^{n + m + 1} \, | \, x \in F(y), \mu \geq \phi(x, y) \}
\end{align*}
is convex, too; $\graph F_\phi$ is the projection of $G$ onto $\mathcal X$, therefore a convex set.
\end{proof}

Recall that $v$ and $v_\phi$ are functions onto the {\em extended} real line $\overline {\mathbb R} \Def= \mathbb R \cup \{\pm \infty\}$ if, as usual, we associate $v(y) = +\infty$ in case that $F(y)$ is empty and $v(y) = -\infty$ in case that $\phi(\cdot, y)$ does not have a minimum on $F(y) \neq \emptyset$. The (effective) domain $\domain v$ is defined as the set of parameters $y \in \mathcal Y$ such that $v(y) < +\infty$ and the function $v$ is called {\em proper} on $Y \subset \mathcal Y$ if it has non-empty domain and $v(y) > -\infty$ for all $y \in Y$.
In the context of optimization, we may also call $\domain v$ the {\em feasible} domain. It follows immediately from the definitions that $v$ (resp., $v_\phi$) is feasible exactly on the domain of $F$ (resp., of $F_\phi$), and that $\domain F = \domain F_\phi$. %Dual infeasibility, on the other hand, is paramount to unboundedness of the optimization problem.

\begin{definition}
	A function $f: \mathcal Y \rightrightarrows \mathbb R$ is {\em locally bounded from below} at $y_0 \in \domain f$ if and only if there exists a constant $m \in \mathbb R$ such that
	\begin{align*}
		\forall \mu \in f(A), \quad m \leq \mu
	\end{align*}
	for some open neighbourhood $A \subset \mathcal Y$ of $y_0$.
\end{definition}

\begin{theorem}
	Let $y_0 \in \interior \domain F_\phi$; the following statements are equivalent:
	\begin{enumerate}
		\item\label{item:dualfeas-bounded} $F_\phi$ is locally bounded from below at $y_0$;
		\item\label{item:dualfeas-feasdom-aux} $v_\phi(\cdot) > -\infty$ on a neighbourhood of $y_0$; %$y_0 \in \interior \domain v_\phi$;
		\item\label{item:dualfeas-feasdom} $v(\cdot) > -\infty$ on a neighbourhood of $y_0$; %$y_0 \in \interior \domain v$.
	\end{enumerate}
	if $F_\phi$ is closed-valued in the neighbourhood of $y_0$.
\end{theorem}
\begin{proof}
	For the implication from \ref{item:dualfeas-bounded} to \ref{item:dualfeas-feasdom-aux}, let $F_\phi$ be locally bounded from below at $y_0$ with constant $m > -\infty$ and neighbourhood $A \subset \mathcal Y$ and assume without loss of generality that $A \subset \domain F_\phi$. For any $y \in A$, we then have that $\emptyset \subsetneq F_\phi(y) \subset [m, +\infty)$ and, since $F_\phi(y)$ is closed, $v_\phi(y) \geq m$. %$v_\phi(y) \in \mathbb R$. Therefore, $A \subset \domain v_\phi$.
	
	We show the implication from \ref{item:dualfeas-feasdom-aux} to \ref{item:dualfeas-feasdom} by contradiction: let $y$ be in the neighbourhood of $y_0$ such that $v_\phi(y) > -\infty$; denote the minimum by $\mu^*$. Assume that $v(y) = -\infty$, that is, for any $x \in F(y_0)$ there exists $x'$ such that $\phi(x', y) < \phi(x, y)$. By definition, however, $\mu^* \geq \phi(x, y)$ for some $x \in F(y)$ and therefore, $\mu' = \phi(x', y) < \phi(x, y) \leq \mu^*$. This contradicts the assumption that $\mu^*$ is the minimum of $F_\phi(y)$ and hence, $v(y) > -\infty$.
	%by contraposition: Let $v(y_0) = -\infty$; that is, for any $x \in F(y_0)$ and $r > 0$ there exists $x'$ such that $\phi(x',  y_0) + r < \phi(x, y_0)$. Take now $\mu \in F_\phi(y_0)$; by definition, $\mu \geq \phi(x, y_0)$ for some $x \in F(y_0)$. Then, however, $\mu' = \phi(x', y_0) < \phi(x, y_0) - r \leq \mu - r$ and $v_\phi(y_0) = -\infty$. This shows that $y_0 \in \domain v_\phi \implies y_0 \in \domain v$; the desired result follows by virtue of the invariance of the interior operator.
	
	For the remaining implication, from \ref{item:dualfeas-feasdom} to \ref{item:dualfeas-bounded}, let $A \subset \domain v$ be a bounded open neighbourhood of $y_0$ such that $v(A) \subset \mathbb R$; then by boundedness of $A$, there exists a constant $m \in \mathbb R$ such that $m \leq v(y)$ for all $y \in A$. In other words, $\phi(x, y) \geq m$ for any $x \in F(y)$ and hence, $\mu \geq  m$ for any $\mu \in F_\phi(y)$. That is, $F_\phi$ is locally bounded from below at $y_0$.
\end{proof}

We leave it to the reader to convince themselves that $v_\phi(y_0) > -\infty$ if and only if $v(y_0) > -\infty$ for $y_0 \in \domain F_\phi$.

\begin{proposition}
	The optimization problems $v$ and $v_\phi$ are equivalent.
\end{proposition}
\begin{proof}
	Let $\bar y \in \domain v$ and $v(\bar y) = \phi(x_0, \bar y)$ for some $x_0 \in F(\bar y)$. Then there exists $\mu \in F_\phi(\bar y)$, namely $\mu = \phi(x_0, \bar y)$, and hence, $v_\phi(\bar y) \leq \mu = v(\bar y)$. Let now $\bar y \in \domain v_\phi$ and $v_\phi(\bar y) = \mu_0$. Then $\mu_0 \in F_\phi(\bar y)$, that is, there exists $x \in F(\bar y)$ with $\mu_0 \geq \phi(x, \bar y)$ and hence, $v(\bar y) \leq \mu_0 = v_\phi(\bar y)$, completing the proof.
\end{proof}

The feasible cost space characterises the auxiliary optimization problem:

\begin{proposition}
	The graph of $F_\phi$ is equal to the epigraph of $v_\phi$.
\end{proposition}
\begin{proof}
	To prove that $\graph F_\phi \subseteq \epigraph v_\phi$, let $\mu \in F_\phi(y)$ for $y \in \mathcal Y$; then $\mu \geq v_\phi(y)$ and hence, $(y, \mu) \in \epigraph v_\phi$. In order to prove that $\epigraph v_\phi \subseteq \graph F_\phi$, take $y \in \mathcal Y$; if $\mu \geq v_\phi(y)$ for some $\mu \in \mathbb R$, then there exists $\mu_0 \in F_\phi(y)$ with $v_\phi(y) \leq \mu_0 \leq \mu$ and $(y, \mu) \in \graph F_\phi$ since $(y, \mu_0) \in \graph F_\phi$.
\end{proof}

The main result of this section follows almost immediately.

\begin{theorem}
	Let $Y \subset \mathcal Y$ be a non-empty, open convex subset of the domain of $F$ and let $F_\phi$ be closed and locally bounded from below on $Y$; then $Y \subset \domain v$ and $v$ is a proper convex function on $Y$.
\end{theorem}
\begin{proof}
	Since $F_\phi$ is closed and locally bounded from below at every $y \in Y$, we have that $v(y) = v_\phi(y) > -\infty$. Furthermore, $y \in \domain v$ and $\graph F_\phi = \epigraph v_\phi$. As $\graph F_\phi$ is closed and convex, so is $\epigraph v$ and hence, $v$ is a proper convex function on $Y$.
\end{proof}

It is easy to see that $F_\phi$ would be locally bounded from below if (but not only if) either $\phi$ is bounded from below or $F$ has compact values.

\begin{lemma}
	Let $f: \mathcal Y \rightrightarrows \mathbb R$ be a convex function and $Y \subset \domain f$ be an open convex set. If $f$ is locally bounded from below at $y_0 \in Y$, then $f$ is  locally bounded on $Y$.
\end{lemma}
\begin{proof}
	By contradiction for $y, y_0 \in Y$; let $f$ be locally bounded from below at $y_0 \in Y$ with constant $m_0 > -\infty$ and neighbourhood $A \subset Y$. Assume that for any $m \in \mathbb R$ there exists $\mu \in f(y)$ with $\mu < m$. Take $y_\alpha = y_0 + \alpha (y_0 - y) \in A$ for some $\alpha > 0$ and pick $\mu_\alpha \in f(y_\alpha)$. Define $\beta = \alpha/(1+\alpha)$ and note that, in particular, there exists $\mu \in f(y)$ such that $\mu < (1 - \beta^{-1}) \mu_\alpha + \beta^{-1} m_0$. Since $f$ is convex and $\beta \in (0, 1)$, we have that $\beta \mu + (1 - \beta) \mu_\alpha \in f(\beta y + (1 - \beta) y_\alpha)$; note
\begin{gather*}
	\beta y + (1 - \beta) y_\alpha = y_0 \\
	\beta \mu + (1-\beta) \mu_\alpha < m_0
\end{gather*}
However, this contradicts the assumption that $m_0$ is a lower bound of $f(y_0)$; hence, $f(y)$ is bounded from below, the desired result.
\end{proof}

It is well known \cite{WayneStateUniversity1972} that a proper convex function from a finite-dimensional vector space onto the real line\footnote{N.B.: Under stronger assumptions, Lipschitzanity holds for mappings between partially ordered, infinite-dimensional vector spaces.} is locally Lipschitz continuous.

\begin{corollary}
	Let $Y \subset \mathcal Y$ be a non-empty open convex subset of $\domain F$ and $F_\phi$ be closed; if $F_\phi$ is locally bounded from below at $y_0 \in Y$, then $v$ is Lipschitz continuous at $y_0$. Moreover, $v$ is uniformly continuous on $Y$. $\lhd$
\end{corollary}

\end{document}